\documentclass[a4paper, 11pt]{amsart}

\usepackage[
            includefoot,  
            marginparwidth=0in,     
            marginparsep=0in,       
            margin=1.45in,               
            includemp]{geometry}

\usepackage{bbm}
\usepackage{float}
\usepackage{graphicx}                  
\usepackage{amssymb}
\usepackage{amsfonts}
\RequirePackage{amsmath}
\RequirePackage{amsthm}

\usepackage{color}
\usepackage{ulem}

\usepackage{fancyhdr}

\newtheorem{thm}{Theorem}[section]
\newtheorem{defn}[thm]{Definition}
\newtheorem{prop}[thm]{Proposition}

\newtheorem{rmq}{Remark}[section]

\newcommand{\eps}{\varepsilon}

\newcommand{\R}{\mathbb{R}}

\newcommand{\Var}{\operatorname{Var}}
\newcommand{\Hess}{\operatorname{Hess}}

\newcommand{\Tr}{\operatorname{Tr}}
\newcommand{\Id}{\operatorname{Id}}

\begin{document}

\title{Bounds on optimal transport maps onto log-concave measures}
\date{\today}

\author[Colombo]{Maria Colombo}
\address{EPFL SB, Station 8, 
CH-1015 Lausanne, Switzerland
}\address{Institute for advanced study, 1 Einstein Dr, Princeton, NJ 08540
}
\email{maria.colombo@epfl.ch}

\author{Max Fathi}
\address{CNRS and Institut de Math\'ematiques de Toulouse\\ Universit\'e de Toulouse, 118 route de Narbonne, 31068, Toulouse
}
\email{max.fathi@math.univ-toulouse.fr}

\begin{abstract} We consider strictly log-concave measures, whose bounds degenerate at infinity. We prove that the optimal transport map from the Gaussian onto such a measure is locally Lipschitz, and that the eigenvalues of its Jacobian have controlled growth at infinity.
\end{abstract}

\maketitle

\section{Introduction and main results}

The goal of this work is to establish \textit{quantitative} regularity bounds on optimal transport maps. The Brenier theorem \cite{Bre91} asserts that given two probability measures on $\R^d$ with density, there exists a convex function $\varphi$ such that $T = \nabla \varphi$ is a transport map sending one measure onto the other. Moreover, it is the optimal map in the Monge-Kantorovitch optimal transport problem with quadratic cost. The qualitative regularity theory has been studied in many works, see \cite{Fig17} and references therein, as well as \cite{CEF18} for the non-compact setting, which is of particular relevance for the present work. 

The Caffarelli contraction theorem \cite{Caf00} states that the Brenier optimal transport map sending the standard Gaussian measure $\gamma$ on $\mathbb{R}^d$ onto a uniformly log-concave measure, that is a measure of the form $e^{-V}dx$ with $\Hess V \geq \alpha \Id$, is $\alpha^{-1/2}$-lipschitz. This result has many applications in probability and analysis, implying for example various functional and isoperimetric inequalities \cite{CEFM04, CFP}, correlation inequalities \cite{CE02}, and bounds on eigenvalues of certain diffusion operators \cite{Mil18}. Alternative proofs have been found \cite{KM12, FGP}, as well as several extensions \cite{CFJ, Kol13}. 

In this work, we are interested in the situation where the target distribution is only log-concave, that is we only assume $\Hess V \geq 0$. In that case, the Brenier map may not be lipschitz anymore. In dimension one, where the map can be explicitly computed via the cumulative distribution functions, it turns out that the derivative of the map has sublinear growth, in a universal way. More precisely, if the target distribution $\mu$ is log-concave and centered, then the Brenier map $T$ sending $\gamma$ onto $\mu$ satisfies
$$|T'(x)| \leq c(1 + |x|)$$
where the constant $c$ only depends on the second moment of $\mu$. This estimate is in general sharp, since that is the precise growth rate of the map onto the symmetrized exponential distribution. We refer to \cite{CFP} for a proof. In this work, we are interested in the situation for higher-dimensional target distributions. Since the class of log-concave measures is invariant by dilations, we need to fix a scaling parameter, so we shall consider isotropic measures, that is probability distributions whose covariance matrix is equal to the identity matrix. 

In a different direction, a result of Klartag and Kolesnikov \cite{KK15} states that eigenvalues of the gradient of the optimal transport map between two log-concave measures concentrate around their average, at a logarithmic scale. That is, if $\lambda_1(x) \leq \lambda_2(x)... \leq \lambda_d(x)$ are the ordered eigenvalues of $\nabla T(x)$, then $\Var(\log \lambda_i) \leq 4$. This provides a dimension-free bound on the fluctuations of the eigenvalues, but does not provide an estimate of their typical size, or pointwise bounds like that of the Caffarelli contraction theorem. 

If we tensorize the one-dimensional result of \cite{CFP}, we can see that for $d$-dimensional isotropic log-concave \textit{product} measures, $||\nabla T||_{op} \leq c(1 + ||x||_{\infty})$. This would be a satisfactory bound, since the right-hand side is almost dimension-free on average (its expectation scales like $\sqrt{\log d}$). However, since the class of isotropic log-concave measures is invariant by rotations, while this estimate is not, we cannot expect it to hold in a pointwise sense for all log-concave isotropic target distributions. 

So instead, one could look for an estimate in $\ell^2$ norm, of the form
\begin{equation} \label{eq_bnd_deriv_conj}
||\nabla T||_{op} \leq c\sqrt{d + |x|^2}.
\end{equation}
where $|x|$ denotes the standard Euclidean 2-norm. The choice of an additive factor $\sqrt{d}$ comes from the fact that when the measure is isotropic, a typical point has its $\ell^2$ norm of order $\sqrt{d}$. Note that compared with the result for product measures, this would behave in a far worse way for typical points. But even though we would prefer a stronger estimate, we do not know how to prove even this one, and only present partial results in that direction anyway.  

Informally, if the derivative grows linearly at infinity, then we would expect the function to grow quadratically. So if \eqref{eq_bnd_deriv_conj} were actually true, we would expect the transport map itself to satisfy a subquadratic estimate. This is the content of the following theorem: 

\begin{thm} \label{thm_estimate_map}
Let $\mu$ be a centered, isotropic, log-concave probability measure on $\R^d$. Then there exists a universal numerical constant $C$ such that the Brenier map sending the standard Gaussian distribution onto $\mu$ satisfies
$$|T(x)| \leq C(d + |x|^2) \qquad \mbox{for every } x\in \R^d.$$
If instead we assume that $\mu$ is centered and satisfies a Gaussian concentration property with constant $\beta$, then 
$$|T(x)| \leq 12\beta^{-1/2}\sqrt{|x|^2 + 17d}.$$ 
\end{thm}

The first part of this theorem actually holds for much more general probability measures, that satisfy a certain tail estimate we shall describe in the next section. About the second part, see Definition \ref{def_conc} for a definition of Gaussian concentration, which is a property weaker than uniform log-concavity, and does not require log-concavity, but which does not hold for all log-concave distributions. 

The typical order of magnitude of the first bound is quite off, since on average the left hand side scales like $d^{1/2}$, while the right hand side scales like $d$. The second estimate has the correct scaling for typical points and high-dimensional measures.

We shall use this bound to then prove a priori regularity estimates on derivatives of $T$. The first strategy will be to revisit Kolesnikov's proof of Sobolev estimates in the uniformly log-concave case. The main difference will be that we can allow for non-uniform lower bounds on the Hessian of the potential. In this direction, we obtain the following bound: 

\begin{thm} \label{thm_est_reg}
Let $T = \nabla \varphi$ be the Brenier map sending the standard Gaussian measure onto $\mu = e^{-V}dx$. Assume that $\mu$ is centered, isotropic, and that for all $x \in \R^d$
$$c_{1}\Id \geq \Hess V(x) \geq \frac{c_2}{d + |x|}\Id,$$
with $c_{1}, c_2
 > 0$. Then there exists a universal constant $C>0$ such that
$$\left|\left|\frac{\partial^2_{ee}\varphi}{\sqrt{d + |x|^2}}\right|\right|_{p+2, \gamma} \leq \frac{C}{c_2}\left(1 + p\frac{\sqrt{c_{1}}}{4\sqrt{d}}\right).$$
\end{thm}

We remark that a growth of $L^p$ norms slower than $O(p)$ also implies that the function has some finite exponential moment. So we would have a pointwise bound of the form  $\partial_{ee}^2\varphi(x) \leq Cr\sqrt{d + |x|^2}$ on the complement of a set of mass smaller than $e^{-C'r}$ for every $r\geq 1$. 

 In this result, as the dimension grows large, the influence of the upper bound $c_1$ vanishes. This matches well with the intuition that only lower bounds on $\Hess V$ matter to get regularity estimates on $\nabla T$, and it may well be that the assumption of an upper bound is a technical artifact of our method of proof. 

\begin{rmq}
Since $|x|$ typically behaves like $d^{1/2}$, it would be more natural to assume a lower bound of the form $\Hess V(x) \geq \beta(d^{1/2} + |x|)^{-1}\Id$. However, our assumption is weaker, and assuming this stronger estimate would not improve the end result in our proof.  
\end{rmq}

Finally, in another direction we use the bound of Theorem~\ref{thm_estimate_map} and an argument in the spirit of Caffarelli contraction theorem \cite{Caf00} to obtain a bound on the growth of the eigenvalues in $L^\infty$, but this time the growth in $|x|$ is not optimal.

\begin{thm} \label{thm_est_reg2}
Let $T = \nabla \varphi$ be the Brenier map sending the standard Gaussian measure onto the probability measure $\mu = e^{-V}dx$. Assume that $\mu$ is centered, isotropic, and that for all $x \in \R^d$
\begin{equation}
\label{eqn:hp-teo3}
c_{1}\Id \geq \Hess V(x) \geq \frac{c_2}{d + |x|}\Id,
\end{equation}
with $c_{1}\geq1$, $c_2
 > 0$. Then there exists a universal constant $C>0$ such that
\begin{equation} \label{eq_bnd_deriv-bad-first_final}
||\nabla T(x)||_{op} \leq \max\left(C\frac{c_1^2 }{c_2^2}, 1 \right)(d + |x|^2)^2.
\end{equation}
If moreover we assume that 
\begin{equation}
\label{eqn:ass-DV}c_{1} \geq |DV(x)|,
\end{equation}
then
\begin{equation} \label{eq_bnd_deriv-bad_final}
||\nabla T(x)||_{op} \leq \max\left(C\frac{c_1^2 }{c_2^2}, 1 \right)(d^{4/3} + |x|^2). 
\end{equation}
\end{thm}

\section{Concentration and displacement}

In this section, we shall give the proof of Theorem \ref{thm_estimate_map} in the log-concave case. The case of measures satisfying a Gaussian concentration inequality is similar, so we omit it. 

To prove the bound of Theorem \ref{thm_estimate_map}, we shall rely on the following concentration estimate, due to Lee and Vempala \cite{LV17}.

\begin{thm}[Concentration estimate for log-concave distributions, \cite{LV17}]
\label{prop_concentration}
Let $\mu$ be a log-concave and isotropic probability measure on $\R^d$. Then for any $1$-lipschitz function $f$ and $r>0$ we have 
$$\mu\left(\left\{ f \geq \int{fd\mu} + cr\right\}\right) \leq \exp\left(-\frac{cr^2}{r + \sqrt{d}}\right)$$
where $c$ is a numerical constant that does not depend on $\mu$ nor $d$.
\end{thm}

This result is an improvement of an earlier theorem of Paouris \cite{Pao06} and Gu\'edon-Milman \cite{GM11}. A longstanding open problem, the Kannan-Lovasz-Simonovits conjecture \cite{KLS95, LV18, CG19} predicts that a better bound of the form $exp(-cr)$ holds, with $c$ universal. 

\begin{proof}[Proof of Theorem \ref{thm_estimate_map}, log-concave case]
Let $x \in \R^n$. Assume that $|T(x)| \geq 8(1 + |x|^2)$ and $|T(x)| \geq 6\sqrt{d}$ (else there is nothing to prove). This also implies that $|T(x)| \geq 8|x|$. Let $u := \frac{T(x) -x}{|T(x) - x|}$ and $B := B(x + 4\sqrt{d}u, 2\sqrt{d})$. Since $T$ is the gradient of a convex function, we know that for any $y \in B$, we have 
\begin{align*}
\langle &T(y) - T(x), y -x \rangle \geq 0 \\
&\Rightarrow T(y) \in T(x) + \{v; \exists w \in B(0,1) \text{ s.t. } \langle v, 2u + w \rangle \geq 0 \} \\
& \Rightarrow T(y) \in T(x) + \{v;  \langle v, u\rangle \geq -|v|/2 \} 
\end{align*}
Now define $f(z) := \langle z - T(x), u \rangle + |z - T(x)|/2$. This function is $3/2$-lipschitz, and the previous computation implies that $T(B) \subset \{f \geq 0 \}$. Moreover, 
\begin{align*}
\int{fd\mu} &= \int{\langle z - T(x), u \rangle + \frac{1}{2}|z - T(x)| d\mu(z)} \\
&= - \langle T(x), u \rangle + \frac{1}{2}\int{|z - T(x)|d\mu(z)} \\ 
&\leq -\frac{|T(x)|^2}{|T(x) - x|} + \frac{\langle x, T(x) \rangle }{|T(x) - x|} + \frac{|T(x)|}{2} + \frac{\sqrt{d}}{2} \\
&\leq -8|T(x)|/9 + 8|x|/7 + |T(x)|/2 + |T(x)|/12 \\
&\leq -|T(x)|(8/9 -1/7 - 1/2 -1/12) \leq -|T(x)|/8.
\end{align*}
We used the fact that $\mu$ is centered and isotropic to estimate 
$$\int{|z|d\mu} \leq \sqrt{\int{|z|^2d\mu}} = \sqrt{d}.$$

Applying Theorem \ref{prop_concentration} to the $1$-lipschitz function $2f/3$, we get
\begin{equation*} 
\mu(T(B)) \leq \mu(\{f \geq 0\}) \leq  \mu\left(\left\{f \geq \int{fd\mu} + \frac{|T(x)|}{8}\right\}\right) \leq \exp\left(-\frac{C|T(x)|^2}{|T(x)| + \sqrt{d}}\right).
\end{equation*}
Since we assumed a priori that $|T(x)| \geq 6\sqrt{d}$, we get
\begin{equation} \label{exp_conc_ball_est}
\mu(T(B)) \leq \exp(-C|T(x)|).
\end{equation}

On the other hand, 
\begin{align*}
\gamma(B) &= \frac{1}{(2\pi)^{d/2}}\int_{B}{e^{-|v|^2/2}dv} = \frac{1}{(2\pi)^{d/2}}\int_{B(0, 2\sqrt{d})}{e^{-|x + 4\sqrt{d}u + v|^2/2}dv}\\
&= \frac{e^{-|x + 4\sqrt{d}u|^2}}{(2\pi)^{d/2}}\int_{B(0, 2\sqrt{d})}{e^{-\langle x + 4\sqrt{d}u, v \rangle - |v|^2/2}dv} \\
&= \gamma(B(0, 2\sqrt{d})e^{-|x + 4\sqrt{n}u|^2/2}\left(\frac{1}{\gamma(B(0, 2\sqrt{d}))} \int_B{e^{-\langle x + 4\sqrt{d}u, v \rangle}d\gamma(v)} \right)\\
&\geq \gamma(B(0, 2\sqrt{d})e^{-|x + 4\sqrt{d}u|^2/2}\exp\left(\gamma(B(0, 2\sqrt{d})^{-1}\int_B{-\langle x + 4\sqrt{d}u, v \rangle d\gamma(v)}\right) \\
&= \gamma(B(0, 2\sqrt{d})e^{-|x + 4\sqrt{d}u|^2/2}. 
\end{align*}
Moreover $\gamma(B(0, 2\sqrt{d})) \geq 3/4$, using Markov's inequality and the fact that $\int{|x|^2d\gamma = d}$. Hence
\begin{equation} \label{gaussian_ball_est}
\gamma(B) \geq \frac{3}{4}e^{-|x + 4\sqrt{d}u|^2/2} \geq e^{-|x|^2 - 17d}. 
\end{equation}
Combining \eqref{exp_conc_ball_est}, \eqref{gaussian_ball_est} and the bound $\gamma(B) \leq \mu(T(B))$, we get
$$|T(x)| \leq C(|x|^2 + d).$$
\end{proof}

In this proof, the log-concavity is only used to derive a concentration bound. Hence we can also get similar estimates under other concentration inequalities, even without log-concavity. The two most classical families of concentration inequalities are the following: 

\begin{defn} \label{def_conc}
Let $\mu$ be a probability measure on $\R^d$. 
\begin{enumerate}
\item $\mu$ is said to satisfy an exponential concentration inequality with constant $\alpha$ if for any $1$-lipschitz function $f$ and any $r \geq 0$ we have
$$\mu\left(\left\{f \geq \int{fd\mu} + r\right\}\right) \leq \exp(-\alpha r);$$

\item $\mu$ is said to satisfy a Gaussian concentration inequality with constant $\beta$ if for any $1$-lipschitz function $f$ and any $r \geq 0$ we have
$$\mu\left(\left\{f \geq \int{fd\mu} + r\right\}\right) \leq \exp(-\beta r^2/2).$$
\end{enumerate}
\end{defn} 

Under these concentration inequalities, we can obtain the following estimates on transport maps: 

\begin{prop} \label{est_gauss_case}
Let $\mu$ be a centered measure on $\R^d$. 

If $\mu$ satisfies the exponential concentration property with constant $\alpha$, then the transport map $T$ sending $\gamma$ onto $\mu$ satisfies  $|T(x)| \leq \max\left(\frac{12}{\alpha}, 8\right)(|x|^2 + 17d)$. 

If $\mu$ satisfies the Gaussian concentration property with constant $\beta$, then the transport map $T$ sending $\gamma$ onto $\mu$ satisfies $|T(x)| \leq \max(12\beta^{-1/2}, 8)\sqrt{|x|^2 + 17d}$. 
\end{prop}

Log-concave distributions satisfy an exponential concentration property, but it is an open problem to get a dimension-free estimate \cite{KLS95, LV18}. The best general lower bound currently known is $cd^{-1/4}$ \cite{Eld13, LV16}, which for our purpose here would be worse than what we obtain using Theorem \ref{prop_concentration}. 

We note that one possible approach to establishing concentration inequalities for log-concave measures is via integrated regularity estimates for transport maps, as proved in \cite{Mil09a}. The approach here is the opposite: we take as given concentration estimates, and seek to deduce new estimates on the Brenier map.

\section{$L^{p}$ estimates on $\nabla T$}

We shall now give a proof of Theorem \ref{thm_est_reg}, which shall follow the strategy of \cite[Section 6]{Kol11}. 

To go from estimates on $T = \nabla \varphi$ to estimates on $\nabla T$, we shall use the Monge-Amp\`ere PDE satisfied by $\varphi$: 
\begin{equation}
(2\pi)^{-d/2}e^{-|x|^2/2} = e^{-V(T(x))}\det \nabla T(x)
\end{equation}
where $\mu(dx) = e^{-V(x)}dx$. 
Hence 
$$\log \det D^2\varphi = -\frac{|x|^2}{2} + V(\nabla \varphi(x)) + C.$$

\begin{proof}[Proof of Theorem \ref{thm_est_reg}]

Assume $c_{1}Id \geq \Hess V(x) \geq c_2(d + |x|)^{-1}\Id$. Note that the upper bound implies that $||(\Hess \varphi)^{-1}||_{op} \leq \sqrt{c_{1}}$ by the Caffarelli contraction theorem applied to the inverse map (in the general form proved in \cite{Kol13}).

To simplify notation, we define $\alpha$ such that $|T(x)| \leq \alpha^{-1}(d + |x|^2)$ and assume $\alpha \leq 1$. 

We introduce the notation $\partial^\eps_{ee} \varphi = \varphi(x + \eps e) + \varphi(x-\eps e) - 2\varphi(x)$ for $e \in \mathbb{S}^{n-1}$ and $\eps > 0$, which is a discretized second-order derivative of $\varphi$. 

From the Monge-Amp\`ere equation, we have
$$|x + \eps e|^2/2 - |x|^2/2 = V(T(x + \eps e)) - V(T(x)) -\log[ \det DT(x)^{-1}\det DT(x+\eps e)].$$

From the lower bound on the Hessian of $V$, we then have
\begin{align} \label{eq_convexity1}
&|x + \eps e|^2/2 - |x|^2/2 \geq \langle T(x+\eps e) - T(x), \nabla V(T(x)) \rangle \\
&+ \frac{\alpha c_2}{4}(d + |x|^2 + \eps^2)^{-1}|T(x+\eps e) - T(x)|^2 -\log[ \det DT(x)^{-1}\det DT(x+\eps e)]. \notag
\end{align}
In this bound, we used the fact that for any $y$ of norm less than $|x| + \eps $, we have $\Hess V(T(y)) \geq \frac{\alpha c_2}{4}(d + |x|^2 + \eps^2)^{-1}\Id$, due to the lower bound on $\Hess V$ and the fact that $|T(y)| \leq 2\alpha^{-1}(d + |x|^2 + \eps^2)$. 

We then multiply this by $(\partial^\eps_{ee}\varphi)^p/(d + |x|^2)^{p/2}$ with $p \geq 0$. Then we get, after a change of variable, integrating by parts and using Lemma 6.3 from \cite{Kol11}, 
\begin{align*}
&\int{\langle T(x+\eps e) - T(x), \nabla V(T(x)) \rangle \frac{(\partial^\eps_{ee}\varphi)^p}{(d + |x|^2)^{p/2}}d\gamma} \\
&= \int{\langle T(T^{-1}(x)+\eps e) - x, \nabla V(x) \rangle \frac{(\partial^\eps_{ee}\varphi)^p\circ T^{-1}}{(d + |T^{-1}(x)|^2)^{p/2}}d\mu} \\
&\geq \int{\left(\Tr(\nabla T(x+\eps e)(\nabla T)^{-1}) \circ T^{-1}-d\right)\frac{(\partial^\eps_{ee}\varphi)^p\circ T^{-1}}{(d + |T^{-1}(x)|^2)^{p/2}}d\mu} \\
&+ p \int{\langle T(T^{-1}(x)+\eps e) - x, \nabla (T^{-1})\nabla \partial^\eps_{ee} \varphi \circ T^{-1} \rangle \frac{(\partial^\eps_{ee}\varphi)^{p-1}\circ T^{-1}}{(d + |T^{-1}(x)|^2)^{p/2}}d\mu} \\
&-p \int{\langle T(T^{-1}(x)+\eps e) - x, \nabla (T^{-1})(x)T^{-1}(x) \rangle \frac{(\partial^\eps_{ee}\varphi)^p\circ T^{-1}}{(d + |T^{-1}(x)|^2)^{p/2+1}}d\mu} \\
&= \int{\left(\Tr(\nabla T(x+\eps e)(\nabla T)^{-1}) -d\right)\frac{(\partial^\eps_{ee}\varphi)^p}{(d + |x|^2)^{p/2}}d\gamma} \\
&+  p \int{\langle T(x+\eps e) - T(x), (\nabla T)^{-1}\nabla \partial^\eps_{ee} \varphi \rangle \frac{(\partial^\eps_{ee}\varphi)^{p-1}}{(d + |x|^2)^{p/2}}d\gamma} \\
&-p \int{\langle T(x+\eps e) - T(x), (\nabla T)^{-1}x \rangle \frac{(\partial^\eps_{ee}\varphi)^p}{(d + |x|^2)^{p/2+1}}d\gamma}.
\end{align*}
Arguing as in \cite[Section 6]{Kol11}, when we substitute the above inequality into \eqref{eq_convexity1} and using the fact that $\Tr(\nabla T(x+\eps e)(\nabla T)^{-1}) -d - \log \det DT(x)^{-1}\det DT(x+\eps e)$ is nonnegative we deduce that
$$\int{(|x + \eps e|^2/2 - |x|^2/2)\frac{(\partial^\eps_{ee}\varphi)^p}{(d + |x|^2)^{p/2}}d\gamma} \geq \frac{\alpha c_2}{4}\int{|T(x+\eps e) - T(x)|^2\frac{(\partial^\eps_{ee}\varphi)^p}{(d + |x|^2 + \eps^2)^{p/2+1}}d\gamma}$$
$$+ p \int{\langle T(x+\eps e) - T(x), (\nabla T)^{-1}\nabla \partial^\eps_{ee} \varphi \rangle \frac{(\partial^\eps_{ee}\varphi)^{p-1}}{(d + |x|^2)^{p/2}}d\gamma}$$
$$-p \int{\langle T(x+\eps e) - T(x), (\nabla T)^{-1}x \rangle \frac{(\partial^\eps_{ee}\varphi)^p}{(d + |x|^2)^{p/2+1}}d\gamma}.$$

We take the symmetric inequality with $-\eps e$ instead of $\eps e$, and we sum them to get
\begin{align*}
&\int{(|x + \eps e|^2/2 + |x - \eps e|^2/2 - |x|^2)\frac{(\partial^\eps_{ee}\varphi)^p}{(d + |x|^2)^{p/2}}d\gamma} \\
&\geq \frac{\alpha c_2}{4}\int{|T(x+\eps e) - T(x)|^2\frac{(\partial^\eps_{ee}\varphi)^p}{(d + |x|^2 + \eps^2)^{p/2+1}}d\gamma}\\
&+\frac{\alpha c_2}{4}\int{|T(x-\eps e) - T(x)|^2\frac{(\partial^\eps_{ee}\varphi)^p}{(d + |x|^2 + \eps^2)^{p/2+1}}d\gamma} \\
&+ p \int{\langle \nabla \partial^\eps_{ee}\varphi, (\nabla T)^{-1}\nabla \partial^\eps_{ee} \varphi \rangle \frac{(\partial^\eps_{ee}\varphi)^{p-1}}{(d + |x|^2)^{p/2}}d\gamma} \\
&-p\int{\langle \nabla \partial^\eps_{ee}\varphi, (\nabla T)^{-1}x \rangle \frac{(\partial^\eps_{ee}\varphi)^p}{(d + |x|^2)^{p/2+1}}d\gamma}.
\end{align*}

We then divide by $\eps^{2p+2}$ and pass to the limit $\eps  \longrightarrow 0$ to obtain
\begin{align*}
&\int{\frac{(\partial^2_{ee}\varphi)^p}{(n + |x|^2)^{p/2}}d\gamma} \geq \frac{\alpha c_2}{2}\int{\frac{(\partial^2_{ee}\varphi)^{p+2}}{(d + |x|^2)^{p/2+1}}d\gamma} \\
&+ p \int{\langle \nabla \partial^2_{ee}\varphi, (\Hess \varphi)^{-1}\nabla \partial^2_{ee}\varphi \rangle \frac{(\partial^2_{ee}\varphi)^{p-1}}{(d + |x|^2)^{p/2}}d\gamma} \\
&-p \int{\langle \nabla \partial^2_{ee}\varphi, (\Hess \varphi)^{-1}x \rangle \frac{(\partial^2_{ee}\varphi)^p}{(d + |x|^2)^{p/2+1}}d\gamma}. 
\end{align*}

For the last term, we have
$$p \int{\langle \nabla \partial^2_{ee}\varphi, (\Hess \varphi)^{-1}x \rangle \frac{(\partial^2_{ee}\varphi)^p}{(d + |x|^2)^{p/2+1}}d\gamma} \leq p \int{\langle \nabla \partial^2_{ee}\varphi, (\Hess \varphi)^{-1}\nabla \partial^2_{ee}\varphi \rangle \frac{(\partial^2_{ee}\varphi)^{p-1}}{(d + |x|^2)^{p/2}}d\gamma}$$
$$+ \frac{p}{4} \int{\langle x, (\Hess \varphi)^{-1}x\rangle \frac{(\partial^2_{ee}\varphi)^{p+1}}{(d + |x|^2)^{p/2+2}}d\gamma}$$
$$\leq  p \int{\langle \nabla \partial^2_{ee}\varphi, (\Hess \varphi)^{-1}\nabla \partial^2_{ee}\varphi \rangle \frac{(\partial^2_{ee}\varphi)^{p-1}}{(d + |x|^2)^{p/2}}d\gamma} + \frac{p\sqrt{c_1}}{4d^{1/2}} \int{ \frac{(\partial^2_{ee}\varphi)^{p+1}}{(d + |x|^2)^{(p+1)/2}}d\gamma}$$
where we used the upper bound on $(\Hess \varphi)^{-1}$. Therefore 
$$\frac{\alpha c_2}{2}\int{\frac{(\partial^2_{ee}\varphi)^{p+2}}{(d + |x|^2)^{p/2+1}}d\gamma}$$
$$\leq \int{\frac{(\partial^2_{ee}\varphi)^p}{(d + |x|^2)^{p/2}}d\gamma} + \frac{p\sqrt{c_1}}{4d^{1/2}} \int{ \frac{(\partial^2_{ee}\varphi)^{p+1}}{(d + |x|^2)^{(p+1)/2}}d\gamma}$$
$$\leq \left(\int{\frac{(\partial^2_{ee}\varphi)^{p+2}}{(d + |x|^2)^{p/2+1}}d\gamma}\right)^{p/(p+2)} + \frac{p\sqrt{c_1}}{4d^{1/2}}\left(\int{\frac{(\partial^2_{ee}\varphi)^{p+2}}{(d + |x|^2)^{p/2+1}}d\gamma}\right)^{(p+1)/(p+2)}.$$

From this bound we can deduce
$$\left|\left|\frac{\partial^2_{ee}\varphi}{\sqrt{d + |x|^2}}\right|\right|_{p+2} \leq \max\left(1, \frac{2}{\alpha c_2}\left(1 + p\frac{\sqrt{c_1}}{4d^{1/2}}\right)\right).$$

\end{proof}

We conclude this Section by stating the analogous result obtained under a stronger Gaussian concentration property: 

\begin{thm}
Let $\mu$ be a centered probability measure satisfying the Gaussian concentration property with constant $\beta$. Assume moreover that
$$c_1 \geq \Hess V(x) \geq \frac{c_2}{d + |x|^2}\Id.$$
Then
$$\left|\left|\frac{\partial^2_{ee}\varphi}{\sqrt{d + |x|^2}}\right|\right|_{p+2} \leq \max\left(1, \frac{C}{ c_2\sqrt{\beta}}\left(1 + p\frac{\sqrt{c_1}}{4}\right)\right).$$
\end{thm}
The proof is exactly the same, we just have to use the stronger bound on $T$ of Theorem \ref{thm_estimate_map} when Gaussian concentration holds. 

\section{$L^\infty$ estimates on $\nabla T$}

{We now move on to the proof of Theorem \ref{thm_est_reg2}, which shall revisit the method used by Caffarelli for the original proof of the contraction theorem. }

\begin{proof}
Convolving the target measure with a smooth, compactly supported convolution kernel $\rho_{\delta}$, we can also assume that the target density is $C^\infty$, up to worsening the constant $c_2$ of a multiplicative factor which is arbitrarily small as $\delta \to 0$. Adding moreover $\delta|x|^2-C(V, \delta)$ to this convoluted potential (here $C(V, \delta)$ is chosen in such a way that the modified potential gives still a probability measure), we can apply Caffarelli's contraction principle to deduce a Lipschitz bound (which degenerates with $\delta$) on the optimal map from the Gaussian.
 The stability of optimal transport maps guarantees that we can reduce to prove our statement in this regularized setting. 
This modification of $V$ also changes the covariance matrix, and therefore the constant prefactor when applying Theorem \ref{thm_estimate_map}, but the effect shall disappear in the limit $\delta \longrightarrow 0$. 
 
Hence, we reduced our statement as follows: we have the additional assumption of the smoothness of $V_\delta$ as well as the modified bounds (in place of \eqref{eqn:hp-teo3})
\begin{equation}
\label{eqn:hp-teo3-new}
c_{1,\delta}\Id \geq \Hess V_\delta(x) \geq \frac{c_2}{d + |x|}\Id+ \delta \Id,
\end{equation}
where $c_{1,\delta} = c_1+\delta$. 
We aim at proving that there exists a universal constant $C>0$ (independent of $\delta$) such that
\begin{equation} \label{eq_bnd_deriv-bad-first}
\partial_{ee} \varphi (x) \leq C \frac{c_{1,\delta}^2 }{c_2^2}(d + |x|^2)^2 \qquad \mbox{for every } e \in \R^d, \; |e|=1
\end{equation}
and that, if we assume that $c_{1,\delta} + 2\delta |x| \geq |DV_\delta(x)|$,
then
\begin{equation} \label{eq_bnd_deriv-bad}
\partial_{ee} \varphi (x) \leq C \frac{c_{1,\delta}^2 }{c_2^2}(d^{4/3} + |x|^2) \qquad \mbox{for every } e \in \R^d, \; |e|=1. 
\end{equation}
With a slight abuse of notation, we will not explicitly denote the dependence of $V_\delta$ and $\varphi$ on $\delta$, but on the other hand $\delta$ can be considered as fixed for the reminder of the proof.

For any function $f: \R^d \to \R$ we consider the (second derivative type) incremental quotient
$$
\partial^\eps_{ee} f(x) = f(x+\eps e)+f(x-\eps e)-2f(x)
$$
Let $1<q,s $ to be fixed at the end of the proof and let $(x_0, \alpha)$ be a maximum point of
$$(x_0 , \alpha) \in {\rm argmax}\Big\{ \frac{\partial^\eps_{ee} \varphi(x)}{(d^{s}+|x|^2)^{\frac q2}} : e \in \R^d , \: |e|=1, \; x_0\in \R^d \Big\}
$$
Since by Caffarelli's theorem \cite{Caf00} the map $\varphi$ has second derivatives bounded by $\delta^{-1/2}$, the previous maximum point is reached at some point $x_0$. Without loss of generality, we may assume that 
$$\max \big\{ {\partial^\eps_{ee} \varphi(x)}{(d^{s}+|x|^2)^{\frac {-q}2}} : e \in \R^d , \: |e|=1, \; x_0\in \R^d \big\} \geq \epsilon^{-2},$$
 since otherwise there would be noting to prove. In this situation, we observe that 
%
at the maximum point $x_0$, $|x_0|$ cannot be too large. More precisely, we have
\begin{equation}
\label{eqn:ubx0} |x_0|^q \leq \delta^{-1/2}.
\end{equation}
Indeed, when $ |x|^q \geq \delta^{-1/2}$ we have that ${\partial^\eps_{ee} \varphi(x)}{(d^{s}+|x|^2)^{\frac {-q}2}} \leq \epsilon^{-2}$, hence such a point $x$ cannot be a maximum.

To shorten notation, in the sequel $\nabla^2 f$ stands for the Hessian of the function $f$. 

At the maximum we have that
$$0= \nabla\Big( \frac{\partial^\eps_{\alpha\alpha} \varphi(x_0)}{(d^{s}+|x_0|^2)^{\frac q2}}  \Big)
=
\frac{\nabla\partial^\eps_{\alpha\alpha} \varphi(x_0)}{(d^{s}+|x_0|^2)^{\frac q2}}
-q \frac{\partial^\eps_{\alpha\alpha} \varphi(x_0) x_0}{(d^{s}+|x_0|^2)^{\frac {q+2}2}},
$$
so that
\begin{equation}
\label{eqn:x0opt1}
\nabla\partial^\eps_{\alpha\alpha} \varphi(x_0) =q \frac{\partial^\eps_{\alpha\alpha} \varphi(x_0) x_0}{(d^{s}+|x_0|^2)}. 
\end{equation}
Moreover, 
\begin{equation*}
\begin{split}
0 \geq  \nabla^2\Big( \frac{\partial^\eps_{\alpha\alpha} \varphi(x_0)}{(d^{s}+|x_0|^2)^{\frac q2}}  \Big)&= \frac{\nabla^2\partial^\eps_{\alpha\alpha} \varphi(x_0)}{(d^{s}+|x_0|^2)^{\frac q2}}- 2q \frac{\nabla\partial^\eps_{\alpha\alpha} \varphi(x_0)\otimes x_0}{(d^{s}+|x_0|^2)^{\frac{q+2}{2}}}
\\&
-q \frac{\partial^\eps_{\alpha\alpha} \varphi(x_0) Id}{(d^{s}+|x_0|^2)^{\frac {q+2}2}} + q(q+2) \frac{\partial^\eps_{\alpha\alpha} \varphi(x_0) x_0 \otimes x_0}{(d^{s}+|x_0|^2)^{\frac {q+4}2}}
\end{split}
\end{equation*}
Using \eqref{eqn:x0opt1} on the second and fourth term in the right-hand side, we can rewrite this bound as 
\begin{equation}
\label{eqn:x0opt2}
\nabla^2\partial^\eps_{\alpha\alpha} \varphi(x_0) \leq q
\frac{\partial^\eps_{\alpha\alpha} \varphi(x_0) Id}{(d^{s}+|x_0|^2)} + q(q-2) \frac{\partial^\eps_{\alpha\alpha} \varphi(x_0) x_0 \otimes x_0}{(d^{s}+|x_0|^2)^{2}}.
\end{equation}

We then take the logarithm of the Monge-Amp\`ere equation 
\begin{equation}
\label{eqn:log-MA}
\log \det \nabla^2\varphi = -\frac{|x|^2}{2} + V_\delta(\nabla \varphi(x)) + C_0.
\end{equation}
Since the map $A \mapsto \log \det A=:F(A)$ is concave, and it holds $\lim_{t\to0}( \det(A+t B) - \det A) /t = {\rm tr} (A^{-1}B)$, we have
$$F(\nabla^2\varphi(x_0\pm \eps \alpha)) = F(\nabla^2\varphi(x_0))+ {\rm tr} [(\nabla^2\varphi(x_0))^{-1}((\nabla^2\varphi(x_0\pm \eps \alpha))- (\nabla^2\varphi(x_0)))].$$
From the previous inequality and \eqref{eqn:x0opt2} we deduce that
$$\partial^\eps_{\alpha \alpha}[F(\nabla^2\varphi) ](x_0) \leq  q
\frac{\partial^\eps_{\alpha\alpha} \varphi(x_0)}{(d^{s}+|x_0|^2) } {\rm tr}\left((\nabla^2 \varphi)^{-1}+(q-2)(\nabla^2 \varphi)^{-1}x_0 \otimes x_0 (d+|x_0|^2)^{-1}\right).
$$
Using \eqref{eqn:log-MA}, the previous inequality can be rewritten as
\begin{equation}
\label{MA-effective}
 \partial^\eps_{\alpha\alpha} [ V_\delta(\nabla\varphi)] (x_0)\leq \epsilon^2 +q
\frac{\partial^\eps_{\alpha\alpha} \varphi(x_0)}{(d^{s}+|x_0|^2) } {\rm tr}\left((\nabla^2 \varphi)^{-1}+(q-2)(\nabla^2 \varphi)^{-1}x_0 \otimes x_0 (d+|x_0|^2)^{-1}\right).
\end{equation}
We define $v:= \nabla\varphi(x_0+\eps \alpha)- \nabla \varphi(x_0)$ if $\max\{ | \nabla\varphi(x_0+\eps \alpha)- \nabla \varphi(x_0)| , | \nabla\varphi(x_0-\eps \alpha)- \nabla \varphi(x_0)|\} = | \nabla\varphi(x_0+\eps \alpha)- \nabla \varphi(x_0)|$ and $v:= \nabla \varphi(x_0)-\nabla\varphi(x_0-\eps \alpha)$ otherwise. Since the computations are essentially the same in both cases, we assume that we are in the first situation.
We observe that since $\varphi$ is convex
\begin{equation*}
\begin{split}
|\partial^\eps_{\alpha\alpha} \varphi(x_0)| &\leq \eps|\nabla\varphi(x_0+\eps \alpha)- \nabla \varphi(x_0-\eps\alpha)|
\\& \leq \eps| \nabla\varphi(x_0+\eps \alpha)- \nabla \varphi(x_0)| + \eps| \nabla\varphi(x_0+\eps \alpha)- \nabla \varphi(x_0)| \leq 2\eps|v|.
\end{split}
\end{equation*}

We rewrite the left-hand side of \eqref{MA-effective} as 
\begin{equation}
\begin{split}
\partial^\eps_{\alpha\alpha} [ V_\delta(\nabla\varphi)] (x_0) = &V_\delta( \nabla \varphi(x_0)+ v) + V_\delta( \nabla \varphi(x_0)- v) -2 V_\delta( \nabla \varphi(x_0)) 
\\&-V_\delta( 2\nabla \varphi(x_0)- \nabla\varphi(x_0+\eps \alpha)) + V_\delta(\nabla \varphi(x_0-\eps \alpha))
\end{split}
\end{equation}
For the first three terms, we have
\begin{equation}
\begin{split}
V_\delta( \nabla &\varphi(x_0)+ v) + V_\delta( \nabla \varphi(x_0)- v) -2 V_\delta( \nabla \varphi(x_0))  
\\&\geq \inf\{ v^T\nabla^2 V_\delta(z)v: {|z| \leq \max\{ | \nabla \varphi(x_0)| + | \nabla \varphi(x_0+\eps \alpha)|+ | \nabla \varphi(x_0+\eps \alpha)|} \}
\\&\geq \inf\{ v^T\nabla^2 V_\delta(z)v: {|z| \leq C((1+|x_0|)^2+d)} \}
\\&\geq \inf\{ v^T\nabla^2 V_\delta(z)v: {|z| \leq 2 C(|x_0|^2+d)} \}
\\& \geq \frac{c_2}{3C(|x_0|^2+d)}|v|^2 \geq \frac{c_2}{12C(|x_0|^2+d)}\eps^{-2}|\partial^\eps_{\alpha\alpha} \varphi(x_0)|^2.
\end{split}
\end{equation}
For the last two terms we employ \eqref{eqn:x0opt1} to obtain
\begin{equation}\label{lasttwo}
\begin{split}
&\left|-V_\delta( 2\nabla \varphi(x_0)- \nabla\varphi(x_0+\eps \alpha)) + V_\delta(\nabla \varphi(x_0-\eps \alpha))\right|
\\&
\leq  \left(\sup\{ |\nabla V_\delta(z)|: {|z| \leq | \nabla \varphi(x_0)|+  | \nabla \varphi(x_0+\eps \alpha)|+ | \nabla \varphi(x_0+\eps \alpha)|} \}\right) \times q \frac{\partial^\eps_{\alpha\alpha} \varphi(x_0) |x_0|}{(d^{s}+|x_0|^2)}
\\&
\leq  \left(\sup\{ |\nabla V_\delta(z)|: {|z| \leq 6C(|x_0|^2+d)} \}\right) \times q \frac{\partial^\eps_{\alpha\alpha} \varphi(x_0)| x_0|}{(d^{s}+|x_0|^2)}.
\end{split}
\end{equation}

Since the map $(\nabla \varphi)^{-1}$ is the optimal transport between $e^{-V_\delta}$ and the Gaussian, thanks to Caffarelli's contraction theorem we deduce that each eigenvalue of the matrix $\nabla(\nabla \varphi)^{-1}=(\nabla^2 \varphi)^{-1}$ is bounded from above by $ c_{1,\delta}^{-1/2}$. Hence, recalling that $q > 1$, the second term in the right-hand side of \eqref{MA-effective} is estimated from above by
\begin{align*}
q\frac{\partial^\eps_{\alpha\alpha} \varphi(x_0)}{(d^{s}+|x_0|^2) } &{\rm tr}((\nabla^2 \varphi)^{-1}+(q-2)(\nabla^2 \varphi)^{-1}x_0 \otimes x_0 (d+|x_0|^2)^{-1}) \\
&\leq q c_{1,\delta}^{1/2} (d+q-2) \frac{\partial^\eps_{\alpha\alpha} \varphi(x_0)}{(d^{s}+|x_0|^2) } 
\\
&\leq q^2 c_{1,\delta}^{1/2} (d+q-2) \frac{\partial^\eps_{\alpha\alpha} \varphi(x_0)}{(d^{s}+|x_0|^2) }\\
&\leq \frac{c_2|\partial^\eps_{\alpha\alpha} \varphi|^2 }{48C\eps^2  (d+|x_0|^2 )} + \frac{12 C\eps^2q^4 c_{1,\delta} d^2 (d+|x_0|^2 ) }{c_2 (d^{s}+|x_0|^2 )^2}
\end{align*}
As regards \eqref{lasttwo}, under the assumption that  $c_{1,\delta} + 2\delta |x| \geq |\nabla V_\delta(x)|$ and thanks to \eqref{eqn:ubx0} we have 
\begin{equation*}
\begin{split}
\sup\{ |\nabla V_\delta(z)|: {|z| \leq 6C(|x_0|^2+d)} \} & \leq c_{1,\delta}+12C \delta (|x_0|^2+d) 
\\&\leq c_{1,\delta}+12C \delta (\delta^{-1/q}+d) = :\bar c_{1, \delta}.
\end{split}
\end{equation*}
We observe that $\bar c_{1, \delta} \to c_1$ as $\delta\to0$. Hence we estimate
\begin{equation}
\begin{split}
\sup\{ |\nabla V_\delta(z)|: {|z| \leq 6C(|x_0|^2+d)} \} &q \frac{\partial^\eps_{\alpha\alpha} \varphi(x_0)| x_0|}{(d^{s}+|x_0|^2)}
 \leq\bar  c_{1,\delta}q |x_0| \frac{\partial^\eps_{\alpha\alpha} \varphi(x_0) }{(d^{s}+|x_0|^2)} 
\\&\leq  \frac{c_2|\partial^\eps_{\alpha\alpha} \varphi|^2 }{
48 C\eps^2 (d+|x_0|^2 )} + \frac{12 C\eps^2c_{1,\delta}^2q^2 |x_0|^2(d+|x|^2 )  }{c_2 (d^{s}+|x_0|^2 )^2}
\end{split}
\end{equation}
Putting together the previous estimates starting from \eqref{MA-effective}, under the assumption \eqref{eqn:ass-DV}  we find that
\begin{equation}
\label{eqn:underest}
 \frac{c_2|\partial^\eps_{\alpha\alpha} \varphi(x_0)|^2 }{
 C\eps^2 (d+|x_0|^2 )} \leq \eps^2+48 C\eps^2q^4 \frac{(c_{1,\delta} d^2+c_{1,\delta}^2 |x_0|^2) (d+|x_0|^2 ) }{c_{2,\delta} (d^{s}+|x_0|^2 )^2}
 \leq  \tilde C \eps^2 q^4 \frac{c_{1,\delta}^2(d^{s}+|x_0|^2 ) }{c_2(d+|x_0|^2 )},
\end{equation}
where $\tilde C$ is a numerical constant, that could be made explicit.
This provides an upper bound on 
$$\max_{x\in \R^d} \frac{c_2|\partial^\eps_{\alpha\alpha} \varphi (x)|^2 }{
 C (d^{s}+|x|^2 )}
 =\frac{c_2|\partial^\eps_{\alpha\alpha} \varphi (x_0)|^2 }{
 C (d^{s}+|x_0|^2 )}
 \leq \tilde C \eps^4 q^4 \frac{c_{1,\delta}^2 }{c_{2,\delta}}
 $$
 by choosing $q=2$ and $s=4/3$ and proves \eqref{eq_bnd_deriv-bad-first} (after letting $\delta$ and then $\eps$ go to zero).
 
 If we don't require assumption \eqref{eqn:ass-DV}, instead, we know from the upper bound on the Hessian that 
 \begin{equation}
 \label{eqn:ub-DV}|\nabla V_\delta(x)| \leq \bar  c_{1,\delta} (\sqrt d+ |x|).
 \end{equation}
 Indeed, we just have to prove that $|\nabla V_\delta(0)| \leq\bar  c_{1,\delta}\sqrt{d}$ and use the upper bound in \eqref{eqn:hp-teo3}. By integration by parts we have
$\int{\nabla V_\delta(x)e^{-V_\delta}dx} = 0$ and by Jensen's inequality and the bound on the second moment of $\mu$
\begin{align*}|\nabla V_\delta(0)| &= |\nabla V_\delta(0) - \int{\nabla V_\delta(x)d\mu}| \\
&\leq \int{|\nabla V_\delta(0)-\nabla V_\delta(x)|d\mu}
\leq\bar  c_{1,\delta}\int{|x|\mu} \leq\bar  c_{1,\delta}(1 + o(\delta))\sqrt{d} ,
\end{align*}
where the $o(\delta)$ comes from the fact that the covariance matrix has been modified when replacing $V$ by $V_{\delta}$, but it converges to the identity matrix as $\delta$ goes to zero. This estimate then implies \eqref{eqn:ub-DV}, up to slightly modifying $\bar  c_{1,\delta}$ in a way that it still converges to $c_1$ in the limit $\delta \longrightarrow 0$.
Hence 
\begin{align*}
\sup&\{ |\nabla V_\delta(z)|: {|z| \leq 6C(|x_0|^2+d)} \} q \frac{\partial^\eps_{\alpha\alpha} \varphi(x_0)| x_0|}{(d^{s}+|x_0|^2)}\\
&\leq 6q\bar  c_{1,\delta} |x_0| (d+|x_0|^2)\frac{\partial^\eps_{\alpha\alpha} \varphi(x_0) }{(d^{s}+|x_0|^2)} 
\\&\leq  \frac{c_{2,\delta}|\partial^\eps_{\alpha\alpha} \varphi|^2 }{
48 C\eps^2 (d+|x_0|^2 )} + \frac{108 C\eps^2\bar  c_{1,\delta}^2q^2 |x_0|^2(d+|x_0|^2 )^2  }{c_{2,\delta} (d^{s}+|x_0|^2 )^2}
\end{align*}
In this case \eqref{eqn:underest} becomes
 \begin{equation}
\label{eqn:underest2}
 \frac{c_{2,\delta}|\partial^\eps_{\alpha\alpha} \varphi(x_0)|^2 }{
 C\eps^2 (d+|x_0|^2 )} \leq \eps^2+4 C\eps^2q^4 \frac{c_{1,\delta}^2 (d+|x_0|^2 )^3}{c_{2,\delta} (d^{s}+|x_0|^2 )^2}
 \leq  \tilde C q^4 \frac{c_{1,\delta}^2(d^{s}+|x_0|^2 ) ^2}{c_{2,\delta}(d+|x_0|^2 )},
\end{equation}
and choosing $q=4$ and $s=1$ concludes the proof of \eqref{eq_bnd_deriv-bad}.
\end{proof}

\textbf{\underline{Acknowledgments}}: M.C. was supported by the SNF Grant 182565 and by the NSF under Grant No. DMS-1638352.
 M.F. was supported by the Projects MESA (ANR-18-CE40-006) and EFI (ANR-17-CE40-0030) of the French National Research Agency (ANR), and ANR-11-LABX-0040-CIMI within the program ANR-11-IDEX-0002-02. Part of this work was done during the Labex CIMI semester program on calculus of variations and probability in the spring 2019.

\end{document}